\documentclass[10pt]{amsproc}
\usepackage{amsmath,amsthm,amsfonts,amscd,float}
\newcommand{\Gal}{{\rm Gal}}
\newcommand{\XT}{\tilde{X}}

\newcommand{\CC}{\mathbb{C}}
\newcommand{\PP}{\mathbb{P}}
\newcommand{\FF}{\mathbb{F}}
\newcommand{\OO}{\mathcal{O}}
\newcommand{\ZZ}{\mathbb{Z}}
\newcommand{\QQ}{\mathbb{Q}}

\newcommand{\Frob}{{\rm Frob}}

\newcommand{\image}{{\rm im}\thinspace}
\newcommand{\Trace}{{\rm tr}\thinspace}
\newcommand{\trace}{{\rm tr}\thinspace}

\newcommand{\Mod}{\negthinspace{\rm mod}\thinspace}
\newcommand{\Fpbar}{\overline{\mathbb{F}}_p}
\newcommand{\Qpbar}{\overline{\mathbb{Q}}_p}

\theoremstyle{plain}
\newtheorem{thm}{Theorem}[section]

\newtheorem{lem}[thm]{Lemma}
\newtheorem*{cor}{Corollary}
\newtheorem{prop}[thm]{Proposition}

\theoremstyle{definition}

\newtheorem{conj}{Conjecture}[section]

\newtheorem*{case1}{Case 1}
\newtheorem*{case2}{Case 2}
\newtheorem*{step1}{Step 1}
\newtheorem*{step2}{Step 2}
\newtheorem*{step3}{Step 3}
\newtheorem*{question}{Question}

\theoremstyle{definition}

\floatstyle{plain}
\newfloat{Program}{thp}{lop}[section]
\newfloat{Table}{thp}{lop}[section]

\begin{document}

\title{A modular quintic Calabi-Yau threefold of level 55}
\author{Edward Lee}
\address{School of Mathematics, Aras Na Laoi, University College Cork, Cork, Ireland}
\email{e.lee@ucc.ie, edward.lee@post.harvard.edu}


 
\begin{abstract}
In this note we search the parameter space of Horrocks-Mumford quintic threefolds and locate a Calabi-Yau threefold which is modular, in the sense that the L-function of its middle-dimensional cohomology is associated to a classical modular form of weight 4 and level 55.
\end{abstract}
\maketitle


\section{Introduction}

As a consequence of Khare and Wintenberger's proof of Serre's conjecture \cite{bib:KW1}, \cite{bib:KW2} on absolutely irreducible odd two-dimensional Galois representations \cite{bib:GY}, any Calabi-Yau threefold over $\QQ$ with the property that the semisimplification of its middle-dimensional cohomology motive splits into two-dimensional irreducible Galois representations is modular; the piece with Hodge weight $(3,0) + (0,3)$ has $L$-function equal (modulo the local $L$-factors at bad primes) to the Mellin transform of a weight 4 modular form, while the remaining pieces of Hodge weight $(2,1) + (1,2)$ have $L$-functions equal to the Mellin transforms of Tate twists of weight 2 modular forms.  This is a natural generalization of the modularity property of elliptic curves (famously proven in \cite{bib:BCDT} and \cite{bib:TW}) to higher-dimensional varieties.  The level of the weight 4 modular form should be divisible by the primes of bad reduction, though it is not known how to compute the exponents of the primes.  Many known examples of modular Calabi-Yau threefolds are catalogued in \cite{bib:Meyer}.

The requirement that the middle-dimensional cohomology motive split is very restrictive.  In the case that the middle cohomology is 2-dimensional, the Calabi-Yau variety is rigid and does not deform.  In all known cases where the middle cohomology splits into 2-dimensional pieces, the splitting is caused by the existence of embedded ruled surfaces over elliptic curves.  As a result, all known examples of Calabi-Yau threefolds over $\QQ$ for which $H^3_{et}$ occur as isolated members of families, and using a deformation theory argument Cynk and Meyer \cite{bib:CY} conjecture that this is always the case.

In this paper, we continue the study of the subfamily of Horrocks-Mumford quintic threefolds that was initiated in \cite{bib:Lee}.  The parameter space of HM-quintics is four-dimensional, and we locate our new modular threefold by imposing natural symmetry and singularity conditions on our candidates.

{\bf Acknowledgement:}  This material is based upon works supported by the Science Foundation Ireland under Grant No. MATF634.


\section{Horrocks-Mumford quintics}

\subsection{Construction of the bundle}

The Horrocks-Mumford vector bundle $HM$ is a stable, indecomposable
rank 2 bundle over the complex projective space $\PP^4$.  It is
essentially the only known bundle satisfying these properties; all
other such bundles that are currently known are derived from $HM$ by
twisting by powers of the sheaf $\OO(1)$ or by taking pullbacks
to branched covers of $\PP^4$.  It was first discovered by Horrocks
and Mumford in \cite{bib:HM}, and has been further studied by many
other authors (see for example \cite{bib:BHM}, \cite{bib:BM},
\cite{bib:HL}, \cite{bib:Schoen}).  In this section we will describe the
construction of $HM$ and explain some of its properties that we will
use later.

The following exposition of the Horrocks-Mumford bundle is taken from
\cite{bib:Hulek}.

A {\em monad} is a three-term complex

\begin{equation*}
\begin{CD}
A @>p>> B @>q>> C
\end{CD}
\end{equation*}

\noindent of vector bundles where $p$ is injective and $q$ is surjective.  The
cohomology of the monad

\[
E = \ker q / \image p
\]

\noindent is also a vector bundle.  To construct the Horrocks-Mumford bundle
using a monad we fix a vector space

\[
V \cong \CC^5.
\]

Denote its standard basis by $e_i, i \in \ZZ/5$.  On the projective
space $\PP^4 = \PP(V)$, we have the Koszul complex

\begin{equation*}
\begin{split}
0 \longrightarrow \OO \stackrel{\wedge s}\longrightarrow V \otimes \OO(1)
&\stackrel{\wedge s}\longrightarrow {\wedge}^2 V \otimes \OO(2) \\
&\stackrel{\wedge s}\longrightarrow {\wedge}^3
V \otimes\OO(3) \stackrel{\wedge s}\longrightarrow {\wedge}^4 V
\otimes \OO(4) \longrightarrow \OO(5) \longrightarrow 0. \\
\end{split}
\end{equation*}

Now the quotient $\OO(1) \otimes V / \OO$ is isomorphic to the tangent
sheaf $T$, and in the Koszul complex the sheaf of cycles
$\image (\OO(i) \otimes \wedge^i V) \subset \OO(i+1) \otimes
\wedge^{i+1}V$ is isomorphic to $\wedge^i T$.  Thus from the map

\[
\OO(2) \otimes \wedge^2 V \longrightarrow \OO(3) \otimes \wedge^3 V
\]

\noindent we obtain the sequence of maps

\[
\OO(2) \otimes \wedge^2 V \stackrel{p_0}\longrightarrow \wedge^2 T \stackrel{q_0}\longrightarrow
\OO(3) \otimes \wedge^3 V
\]

\noindent where the first map is surjective and the second is
injective.

Horrocks and Mumford defined the following maps

\begin{equation*}
\begin{aligned}
f^+ &: V \longrightarrow \wedge^2 V, & f^+ (\sum v_i e_i) &= \sum v_i e_{i+2}
\wedge e_{i+3}, \\
f^- &: V \longrightarrow \wedge^2 V, & f^- (\sum v_i e_i) &= \sum v_i e_{i+1}
\wedge e_{i+4}. \\
\end{aligned}
\end{equation*}

Using these maps one can define

\begin{equation*}
\begin{aligned}
p&: V \otimes \OO(2) \stackrel{(f^+, f^-)(2)}\longrightarrow 2 \wedge^2 V
\otimes \OO(2) \stackrel{2 p_0}\longrightarrow 2 \wedge^2 T \\
q&: 2 \wedge^2 T \stackrel{2 q_0}\longrightarrow 2 \wedge^3 V \otimes \OO(3)
\stackrel{(-f^{-*}, f^{+*})(3)}\longrightarrow V^* \otimes \OO(3).
\end{aligned}
\end{equation*}

One easily checks that $q \circ p = 0$.  Hence we obtain a monad

\begin{equation*}
V \otimes \OO(2) \stackrel{p}\longrightarrow 2 \wedge^2 T
\stackrel{q}\longrightarrow V^* \otimes
\OO(3).
\end{equation*}

Its cohomology

\begin{equation*}
HM = \ker q / \image p
\end{equation*}

\noindent is the Horrocks-Mumford bundle.  It is a rank 2 bundle, and its total
Chern class $c(HM)$ equals $c(\wedge^2 T)^2 c(V^* \otimes \OO(3))^{-1}
c(V \otimes \OO(2))^{-1}$.  Using the splitting principle, one
computes this class to be $1 + 5H + 10 H^2$.  Therefore, zero sets of
sections of $HM$ are surfaces of degree 10; Horrocks and Mumford
showed that the generic zero set is a smooth abelian surface.

\subsection{Symmetries of $HM$ and invariant quintics}

The study of $HM$ is greatly expedited by the fact that it
admits a large group of discrete symmetries.  Consider the Heisenberg
group of rank 5, which we denote by $H_5$.  We present it as a
subgroup of $GL_5(\CC)$ generated by the matrices

\[
\sigma = \begin{pmatrix} & 1 & & & \\ & & 1 & & \\ & & & 1 & \\ & & &
  & 1 \\ 1 & & & & \end{pmatrix}, \tau = \begin{pmatrix} 1 & & & & \\
  & \epsilon & & & \\ & & \epsilon^2 & & \\ & & & \epsilon^3 & \\ & &
  & & \epsilon^4 \end{pmatrix},
\]

\noindent where $\epsilon = e^{\frac{2 \pi i}{5}}$ is a primitive fifth root of
unity.  $H_5$ is a central extension

\[
1 \rightarrow \mu_5 \rightarrow H_5 \rightarrow \ZZ/5 \times \ZZ/5
\rightarrow 1
\]

\noindent where $\sigma$ is sent to $(1,0)$ and $\tau$ to $(0,1)$.
Here $\mu_5$ is the multiplicative group of fifth roots of unity.

In fact, the normalizer $N_5$ of $H_5$ in $SL_5(\CC)$ preserves $HM$.
$N_5$ is a semidirect product of $H_5$ with the binary icosahedral
group $SL(2,\ZZ_5)$.  We will need the following elements of $N_5$:

\[
\iota = \begin{pmatrix} 1 & & & & \\ & & & & 1 \\ & & & 1 & \\ & & 1 &
  & \\ & 1 & & & \end{pmatrix}, \mu = \begin{pmatrix} 1 & & & & \\ &
  & 1 & & \\ & & & & 1 \\ & 1 & & & \\ & & & 1 & \end{pmatrix},
  \delta = \frac{1}{\sqrt{5}} \begin{pmatrix} 1 & 1 & 1 & 1 & 1 \\ 1 & \epsilon & \epsilon^2 & \epsilon^3 & \epsilon^4 \\ 1 & \epsilon^2 & \epsilon^4 & \epsilon & \epsilon^3 \\ 1 & \epsilon^3 & \epsilon & \epsilon^4 & \epsilon^2 \\ 1 & \epsilon^4 & \epsilon^3 & \epsilon^2 & \epsilon
  \end{pmatrix}.
\]

En route to determining the sections of $HM$, Horrocks and Mumford
determined the $N_5/H_5$-module $\Gamma_{H_5}(\OO(5))$ of
$H$-invariants of $\Gamma(\OO(5))$, i.e. Heisenberg-invariant quintics
in $\PP^4$.  It is six-dimensional, spanned by the polynomials

\[
\sum x_i^5, \sum x_i^3 x_{i+1} x_{i+4}, \sum x_i x_{i+1}^2 x_{i+4}^2,
\]
\[
\sum x_i^3 x_{i+2} x_{i+3}, \sum x_i x_{i+2}^2 x_{i+3}^2, x_0 x_1 x_2
x_3 x_4
\]

\noindent where the sums are taken over powers of $\sigma$.  The base locus of
this space of quintics is the set of 25 lines $L_{ij}$,
where

\[
L_{00} = \{ x \in \PP^4 : x_0 = x_1 + x_4 = x_2 + x_3 = 0 \},
\]

\[
L_{ij} = \sigma^i \tau^j L_{00}.
\]

Since $c(HM) = 1 + 5H + 10H^2$, $c_1(\wedge^2(HM)) = 5H$ and thus
$\wedge^2(HM) \cong \OO(5)$.  Hence if $s_1$ and $s_2$ are sections of $HM$, the zero set
of the section $s_1 \wedge s_2$ of $\wedge^2{HM}$ is a (singular) quintic
Calabi-Yau threefold that has the structure of a pencil of abelian
surfaces.

\begin{prop}  For generic sections $s_1$ and $s_2$ of $HM$, the
singularities of the resulting threefold are the 100 nodes coming from
the intersection of $Z(s_1)$ and $Z(s_2)$.
\end{prop}

\begin{proof}  For a generic choice of $s_1$ and $s_2$,
$Z(s_1)$ and $Z(s_2)$ intersect transversely in 100
points; these points form the base locus of the pencil.

Let $p$ be a point at which $Z(s_1)$ and $Z(s_2)$ intersect
transversely.  Choose a trivialization of $HM$ near $p$, and put $s_1
= (s_{11}, s_{12})$ and $s_2 = (s_{21}, s_{22})$ relative to this
trivialization.  Since $Z(s_1)$ and $Z(s_2)$ intersect transversely, we may then use $s_{11}, s_{12}, s_{21}$ and $s_{22}$
as local coordinates on $\PP^4$ centered at $p$.  The local equation
for the threefold is then

\[
s_{11} s_{22} - s_{12} s_{21} = 0.
\]

\noindent Hence $p$ is a node.
\end{proof}

Nodes on threefolds result from the vanishing of an $S^3$ cycle on a a smooth 
family of threefolds.  One expects that degenerating the $S^3$ cycles and then 
resolving the singularities will cause the Betti number $h^3$ to drop.  We will 
be interested in birationally equivalent Calabi-Yau threefolds with low Betti number.  
Taking one-parameter families of abelian surfaces in $\PP^4$ gives us a quick way of
manufacturing nodal Calabi-Yau threefolds, whose singularities can
then be resolved.  In \cite{bib:Schoen}, Schoen studied the Fermat
quintic $Q$ defined by the equation 
$$x_0^5 + x_1^5 + x_2^5 + x_3^5 + x_4^5 - 5 x_0 x_1 x_2 x_3 x_4
= 0.$$  Schoen showed that it was a Horrocks-Mumford quintic with 125 nodes
instead of the usual 100, and he proved that the blowup $\tilde{Q}$ of
$Q$ was rigid and modular.  Other nodal Calabi-Yau threefolds whose
resolutions are modular were studied in \cite{bib:WvG}.

Although we have defined the Horrocks-Mumford bundle only over $\CC$, the
pencils of abelian surfaces it defines are quintics in $\PP^4$, and
the quintics we are interested in have integer coefficients and can
thus be studied over arbitrary fields $k$.

It is well-known \cite{bib:Aure} that Horrocks-Mumford quintics are determinantal quintics.  Let $y$ be a generic point in $\PP^4$, and define the matrices $M_y(x) = \{ y_{3(i-j)} x_{3(i+j)} 
\}$ and $L_y(z) = \{y_{i-j} z_{2i-j} 
\}$; observe that $M_y(x) z = L_y(z) x$ when $x$ and $z$ are interpreted as column vectors.  It is shown in 
\cite{bib:Moore} that the threefolds $X_y = 
\{\det M_y(x) = 0 \}$ in $\PP^4(x)$ and $X'_y = 
{ \det L_y(z) = 0 
}$ in $\PP^4(z)$ are Horrocks-Mumford quintics; the equation $M_y(x) z = 0$ defines a threefold $\tilde{X}_y$ in $\PP^4(x) 
\times 
\PP^4(z)$ which is a common partial resolution.  For a generic choice of $y$, there exists a point $y'$ such that the nodes of $X_y$ are the Heisenberg orbit of $\{y, \iota y, y', \iota y' \}$.  Furthermore, $(\sigma, \sigma)$ and $(\tau, \tau^2)$ are automorphisms of $\tilde{X}_y$ lifting the respective automorphisms of $X_y$ and $X'_y$.

In \cite{bib:GP} Gross and Popescu consider symmetric HM-quintics over $\CC$, i.e. HM-quintics for which $y \in \PP^2_{+}$, the positive eigenspace of $\iota$.  Among other results, they show that for generic $y \in \PP^2_{+}$ the singular locus of $X'_y$ consists of a pair of elliptic quintic normal curves $E_{1,y}$ and $E_{2,y}$, and that the singular locus of $\tilde{X}_y$ consists of 50 points lying over $E_{1,y}$ and $E_{2,y}$.  For generic $y$, denote by $\hat{X}_y$ the blow-up of $\tilde{X}_y$.

In an effort to cut down $h^{2,1}$ further, we will search for symmetric HM-quintics containing extra nodes.  By Heisenberg symmetry, imposing the condition that a point $x'$ be a node of $X_{y}$ or $X'_{y}$ means that we get the entire Heisenberg orbit of $x'$ for free.  Since it seems unlikely that we will find a symmetric HM-quintic containing 25 additional nodes, we search for quintics containing the points $(1:0:0:0:0)$ and $(1:1:1:1:1)$ as nodes.  Simple computations yield the following results:

\begin{lem}
Given $y \in \PP^2_+$, the point $(1:0:0:0:0)$ is a singular point of $X_y$ over $\CC$ if and only of one of the coordinates of $y$ is zero.

For $y = (0:1:-1:-1:1), (0:2:3 \pm \sqrt{5} : 3 \pm \sqrt{5} : 2), (2:-1:0:0:-1), (2:\pm \sqrt{5}-1:0:0:\pm \sqrt{5}-1) \in \PP^2_+$, $X_y$ contains the Heisenberg orbits of $(1:0:0:0:0)$ and $(1:1:1:1:1)$ as nodes over $\CC$. $\Box$
\end{lem}

The threefold $X_{(0:1:-1:-1:1)}$ was studied in \cite{bib:Lee}, where it was shown that its middle-dimensional cohomology split into two 2-dimensional pieces coming from ruled elliptic surfaces and a 2-dimensional piece corresponding to the unique normalized modular cusp form of weight 4 and level 5.  Here we will focus our attention on the parameter $y = (2:-1:0:0:-1)$ (the threefolds $X_{(0:2:3:\pm \sqrt{5} : 3 \pm \sqrt{5} : 2)}$ are in fact isomorphic to $X_{(2:-1:0:0:-1)}$ over the field $\QQ(\sqrt{5})$ via the morphisms $(\delta, \delta^{-1})$ and $(\delta^{-1}, \delta)$, suggesting that their associated modular forms are twists of each other by a quadratic character).

For $y = (2:-1:0:0:-1)$, our quintics $X$ and $X'$ are given by the equations

\begin{equation}
\begin{aligned}
X &= \Big\{ \big(\sum_{i \in \ZZ/5} x_i x_{i+2}^2 x_{i+3}^2 - 4 x_i^3 x_{i+2} x_{i+3} \big) - 15 x_0 x_1 x_2 x_3 x_4 = 0 \Big\}, \\
X' &= \Big\{ \big(\sum_{i \in \ZZ/5} z_i^3 z_{i+2} z_{i+3} - 4 z_i z_{i+1}^2 z_{i+4}^2 \big) + 15 z_0 z_1 z_2 z_3 z_4 = 0 \Big\}. \\
\end{aligned}
\end{equation}

\subsection{Singularities of $\tilde{X}$}

Henceforth assume $y = (2:-1:0:0:-1)$ and drop it as a subscript.  Now that we understand the singularities of $\tilde{X}$ over $\CC$, we need to investigate the singularities over a field of general characteristic:

\begin{prop}
Over a field of characteristic not equal to 2, 5 or 11, $\tilde{X}$ has 65
singular points:  the Heisenberg orbits of $(2:-1:0:0:-1) \times (0:1:0:0:-1), (2:-1:0:0:-1) \times (0:-2:1:1:2), (1:0:0:0:0) \times (0:0:1:0:0), (1:0:0:0:0) \times (0:0:0:1:0)$ and $(1:1:1:1:1) \times (1:1:1:1:1)$.  The 65 singular
points of $X$ are all ordinary double points (nodes).\label{theorem:main}
\end{prop}

The proof of the proposition will require the use of several {\em ad hoc} arguments and heavy use of computations in Macaulay2 in order to obtain results valid in general characteristic.  A more conceptual proof would be possible if, for example, we were able to show that $\tilde{X}$ were birational to another threefold such as a fiber product of elliptic surfaces as in \cite{bib:HV}.

We will begin the proof of the proposition after we dispense with some preliminary
results.  Fix a parameter $y \in \PP^2_{+}$.  If $x$ is a point in $\PP^4(x)$, we say that $x$ is a rank
$n$ point if rank $M(x)$ = $n$.  Similarly, if $z$ is a point in
$\PP^4(z)$, we say that $z$ is a rank $n$ point if rank $L(z)$ = $n$.

\begin{lem}  Points of $X$ or $X'$ of rank less than 4 are singular
  points.
\end{lem}

\begin{proof}  More generally, let $A$ be an $n$ by $n$ matrix
whose $ij$ entry is the indeterminate $x_{ij}$.  Using the Laplace
expansion formula, we see that $\frac{\partial \det(A)}{\partial x_{ij}}$
is equal to $(-1)^{i+j} \det A_{ij}$, where $A_{ij}$ is the $ij$ minor
of $A$.

Now suppose that the $x_{ij}$ are functions of some other
indeterminates $y_k$.  By the chain rule, $\frac{\partial
  \det(A)}{\partial y_k}$ = $\sum^n_{i = 1} \sum^n_{j = 1} (-1)^{i+j}
\det A_{ij} \frac{\partial x_{ij}}{\partial y_k}$.  If $y$ is a point
of rank less than $n$, then the determinants of all the $(n-1)$ by
$(n-1)$ minors are zero, and hence $\frac{\partial \det(A)}{\partial
  y_k}$ is zero for all $k$.
\end{proof}

\begin{lem}  The projections $\pi_1: \tilde{X} \rightarrow X$ and $\pi_2:
  \tilde{X} \rightarrow X'$ are isomorphisms on the rank 4 loci of
  $X$ and $X'$ respectively.
\end{lem}

\begin{proof}  If $x$ is a rank 4 point of $\PP^4(x)$, this
means that the kernel of $M(x)$ is 1-dimensional.  Hence the kernel of
$M(x)$ defines a unique point $z$ in $\PP^4(z)$.  We thus have a
regular map $f_1 : X - X_4 \rightarrow \pi^{-1}(X - X_4)$.

Now given a point $(x,z)$ in $X$ such that $x$ is a rank 4 point,
the coordinates of $z$ are given by the determinants of 4 by 4 minors
of $L(x)$.  Hence the map $\pi_1: \pi^{-1}(X - X_4) \rightarrow (X -
X_4)$ is also regular, proving the result for $X$.  A similar result
holds for $X'$.
\end{proof}

\begin{lem}
Assume the characteristic of our base field is not 2.  If $x$ is a point of $F$, then $x$ is a rank 4 point or a rank 3 point.
\end{lem}

\begin{proof}  If $x$ is a point of $F$, then $\det M(x) =
0$.  Hence the rank of $\det M(x)$ is at most 4.

The 3 by 3 minor

\[
\begin{pmatrix}
 2x_0 &  & -x_1 \\  & 2x_1 &  \\ -x_1 &  & 2x_2 
\end{pmatrix}
\]

has determinant $8x_0 x_1 x_2 - 2 x_1^3$.  Its cyclic permutations are also minors; requiring that these polynomials vanish implies that no $x_i$ is zero.  However, the 3 by 3 minor

\[
\begin{pmatrix}
2x_0 &  & -x_4  \\  & 2x_1 & -x_2 \\ -x_1 &  &
\end{pmatrix}
\]

has determinant $-2 x_1^2 x_4$ and thus cannot be zero.  Hence $x$ cannot be a rank 3 point.
\end{proof}

\begin{lem}
If $z$ is a point of $X'$, then $z$ is a rank 4 point or a rank 3
point.
\end{lem}

\begin{proof}  Again, if $z$ is a point of $G$, then $\det
L(z) = 0$.  Hence the rank of $\det L(z)$ is at most 4.

The 3 by 3 minor

\[
\begin{pmatrix}
 &  & -z_3 \\ -z_3 & & \\ 2z_2 & -z_4  \\ 
\end{pmatrix}
\]

has determinant $z_3^2 z_4$; it and its cyclic permutations cut out the $I_5$ lines consisting of the line $\{ z_0 = z_2 = z_3 = 0 \}$ and its Heisenberg orbit.  It is easy to check that for points on these lines the rank of $L(z)$ is equal to 3.
\end{proof}

We can now classify the singularities of $\tilde{X}.$

\begin{proof}[Proof of proposition \ref{theorem:main}.]
A point $(x,z)$ of $\tilde{X}$ is
a singular point if and only if the kernel of the matrix
$\begin{pmatrix} L(z) & M(x) \end{pmatrix}$ has dimension at least 6,
i.e. if the rank is at most 4.  This is equivalent to saying that the
rank of the transposed matrix $\begin{pmatrix} L^T(z) \\ M(x)
\end{pmatrix}$ is at most 4, which is equivalent to saying that the
kernel of $\begin{pmatrix} L^T(z) \\ M(x) \end{pmatrix}$ has dimension
at least 1, i.e. the kernel of $L^T(z)$ and the kernel of $M(x)$ have
nontrivial intersection.

\begin{case1}  $x$ is a rank 4 point.  
Since $(x,z)$ is a
point of $X$, $z$ is in the kernel of $M(x)$.  Therefore $z$ must
span the space $\ker L^T(z) \cap M(x)$.  Hence $L^T(z) z = 0$.  In this case, we have $2(z_0^2 - z_1 z_4) = 0$, as well as its cyclic permutations.  hence $z = (1:\epsilon^j:\epsilon^{2j}:\epsilon^{3j}:\epsilon^{4j})$ for some $j$.  From this we deduce that $x = (1:\epsilon^j:\epsilon^{2j}:\epsilon^{3j}:\epsilon^{4j})$ as well.
\end{case1}

\begin{case2}
$x$ is a rank 3 point.
\end{case2}
Using Macaulay2, we can show that if $(x,z)$
is a singular point of $\tilde{X}$ and $x$ is a rank 3 point of $F$, then $550 
\prod x_i z_i = 0$.  Since we assume the characteristic of our field is not 2, 5 or 11, some $x_i$ or some $z_i$ is zero.

Suppose that some $x_i$ is zero; without loss of generality assume $x_0 = 0$.  Examining the $4 \times 4$ minors of $M(x)$, we must have $x_1 = 0$ or $x_4 = 0$; again without loss of generality assume $x_1 = 0$.  Further calculations show that the rest of the $x_i$ must be nonzero and that $x_3^2 - 4 x_2 x_4 = 0$.

The conditions that $M(x) z = 0$ imply that$z_3 = -x_4 z_0 - x_2 z_1 = 2x_2 z_2 - x_3 z_4 = 0$.  

Suppose that $z_2 = 0$; then $z_4 = 0$ also.  We eventually deduce that $x$ is of the form $(0:0:1:x_3:\frac{x_3^2}{4})$ and $z$ is of the form $(1:-
\frac{x_3^2}{4}:0:0:0)$.  The rank condition on $[M(x) | L(z)]$ implies that $x_3 = -2 \epsilon^k$, where $\epsilon$ is a 5th root of unity.  We can conclude that the point $(0:0:1:-2:1), (1:-1:0:0:0)$ and its Heisenberg orbit are singular points.

Now suppose that $z_2 \neq 0$; this implies that $z_4 \neq 0$ as well.  If $z_0 = 0$, then the equation $M(x) z = 0$ implies that $z_1 = 0$, and the rank condition on $[L(z) M(x)]$ implies that $z = 0$, a contradiction.  With the aid of Macaulay2, we can show that $(x,z)$ has the form $(0:0:1:-2 \epsilon:\epsilon^2), (1:-\epsilon^2:2 \epsilon^4:0:-2 \epsilon^3)$.  Hence the point $(0:0:1:-2:1),(1:-1:2:0:-2)$ and its Heisenberg orbit are singularities as well.

Finally, suppose that no $x_i$ is zero and that $z_0 = 0$.  If either $z_1$ or $z_4$ is zero as well, then the condition $M(x) z = 0$ implies that $z = 0$, which is impossible.  Hence none of the other $z_i$ are 0.

Case 1.  The rank of $z$ is 4.  This implies that $z$ is in fact a singular point of $G$.  But one can check in Macaulay2 that $G$ has no singular points with exactly one zero coordinate.

Case 2.  The rank of $z$ is 3.  This implies that the 4 by 4 minors of $L(z)$ vanish.  Again using Macaulay2, we get $z = (0:2 \epsilon : -1 \epsilon^2 : \epsilon^3 : -2 \epsilon^4)$, and we recover the same points we found when we assumed that some $x_i = 0$.  Hence we have found all the singular points of $\tilde{X}$.
\end{proof}

There are 65 singular points on $\tilde{X}$: those in the Heisenberg orbit of $(1:0:0:0:0) \times (0:0:1:0:0)$ and $(1:0:0:0:0) \times (0:0:0:1:0)$ we call $\sigma$-nodes.  Those in the Heisenberg orbit of $(1:1:1:1:1) \times (1:1:1:1:1)$ we call $\tau$-nodes.  The other 50 nodes move within the entire family of symmetric HM-quintics; we call them regular nodes.

To prove that the 65 singular points are indeed nodes, we can construct local coordinates around each singular point, exhibit $\tilde{X}$ locally as the zero set of a single polynomial and show that the resulting Hessians have nonvanishing determinant; this is a routine but tedious task which we omit here.

For arithmetic purposes, we need to know the value of the Hessian at each node; the rulings of the exceptional $\PP^1 
\times \PP^1$ divisor are defined over $\FF_p$ exactly when the Hessian is a square in $\FF_p$.  We get the following result:

\begin{lem}
The rulings over the $\sigma$-nodes and regular nodes are defined over $\FF_p$ whenever the nodes themselves are; the $\tau$-nodes have rulings defined over $\FF_p$ whenever the nodes are defined and 5 is a square in $\FF_p$.
\end{lem}

Denote by $\hat{X}$ the blowup of $\tilde{X}$ along the union of the 65 nodes' $\hat{X}$ then has good reduction away from the primes $2, 5, 11$.

The $z$-coordinates of our regular nodes should lie on elliptic quintic normal curves; quick computations show that the Heisenberg orbit of $(0:1:0:0:-1)$ lies on the degenerate $I_5$ curve cut out by $\{z_i z_{i+2} \, : 
\, i \in \ZZ/5 \}$, while the Heisenberg orbit of $(0:-2:1:1:2)$ lies on the elliptic curve cut out by $\{2 z_i^2 - z_{i+1} z_{i+4} + 4 z_{i+2} z_{i+3} \, : \, i \in \ZZ/5 \}$; label these curves $E_1$ and $E_2$ respectively.

\begin{lem}
The rank 3 locus of $G$ is precisely the union of the two elliptic quintic normal curves $E_1$ and $E_2$.
\end{lem}

\begin{proof}
Using Macaulay2, it is easy to show that the ideal defined by the 4 by 4 minors of $L(z)$ is contained in both the ideal defining $E_1$ and the ideal defining $E_2$.  Hence both $E_1$ and $E_2$ are contained in the rank 3 locus.  For the converse, Macaulay  shows that for any $i,j$, the element $550 z_i^5 z_{i+2}^5 (2 z_j^2 - z_{j+1} z_{j+4} + 4 z_{j+2} z_{j+3})$ is contained in the ideal defined by the 4 by 4 minors of $L(z)$.  Over a field of characteristic different from 2, 5 and 11, these elements cut out the union of $E_1$ and $E_2$.  Hence the rank 3 locus is contained in $E_1 
\cup E_2$.
\end{proof}

\begin{lem}  If the characteristic of the base field is not 2, 5 or 11, $E_2$ is an elliptic normal curve with discriminant $-2750$, conductor $550$ and Weierstrass equation

\[
y^2 = x^3 - 27(-5909375) x - 54(-8087890625).
\]

\end{lem}

\begin{proof} This was proven by Fisher in Chapter 4 of
\cite{bib:Fisher}.  Assuming the characteristic of the base field is
not 5, Fisher constructs the universal curve $\mathcal{X}(5) \subset
X(n) \times \PP^4$ as the closure of the scheme defined by the 4 by 4
Pfaffians of the 5 by 5 matrix 

\[
\begin{pmatrix}
         & -a_1 x_1 & -a_2 x_2 &  a_2 x_3  &  a_1 x_4 \\
 a_1 x_1 &          & -a_1 x_3 & -a_2 x_4  &  a_2 x_0 \\
 a_2 x_2 &  a_1 x_3 &          & -a_1 x_0  & -a_2 x_1 \\
-a_2 x_3 &  a_2 x_4 &  a_1 x_0 &           & -a_1 x_2 \\
-a_1 x_4 & -a_2 x_0 &  a_2 x_1 &  a_1 x_2  &          \\
\end{pmatrix}
\]

\noindent where $a = (0:a_1:a_2:-a_2:-a_1)$.  $E_2$ and $E_2$ are the
curves obtained when $a = (0:2:-1:1:-2)$.  The fibers are smooth elliptic
normal curves exactly when $a_1 a_2 (a_1^{10} - 11 a_1^5 a_2^5 - a_2^{10})
\neq 0$, in analogy with the case where the base field is $\CC$.  Fisher also derives a formula for the Weierstrass equation of an elliptic normal curve, from which the conductor can be computed.
\end{proof}

By the modularity theorem for elliptic curves, $E_2$ is associated to a modular form of weight 2 and level 550; counting points shows that the correct modular form is $550K1$ in Stein's database.

\subsection{Topology of $\hat{X}$}

We need to compute the topological invariants of $\hat{X}$.
First let $Y$ be a smooth deformation of $\tilde{X}$; that is, let
$Y$ be a smooth threefold obtained by intersecting $\PP^4 \times
\PP^4$ with five divisors of type $(1,1)$.  Topologically,
$\hat{X}$ is obtained from $Y$ by contracting 65 copies
of $S^3$ and replacing them with 65 copies of $\PP^1 \times \PP^1$.
Therefore, we have 
$$\chi((\hat X)) = \chi(Y) + 65
\chi(\PP^1 \times \PP^1) = \chi(Y) + 260.$$

Standard techniques (see \cite{bib:Lee}) show that $\chi(Y) = -100$.  Hence
$\chi(\hat{X}) = 160$.

We can also compute most of the Hodge numbers of $\hat{X}$.
Since $\hat{X}$ is obtained from $Y$ by the surgery
procedure explained above, we have $h^{0,0} = h^{3,3} = 1, h^{1,0} =
h^{0,1} = h^{2,0} = h^{0,2} = 0$, and $h^{3,0} = h^{0,3} = 1$.  The
only unknown Hodge numbers are $h^{1,1} = h^{2,2}$ and $h^{1,2} =
h^{2,1}$, and we know that $2h^{1,1} - h^{2,1} = 160.$


\section{Proof that $\hat{X}$ is modular}

\subsection{The Lefschetz theorem and point-counting}

In \'etale cohomology, we have the Lefschetz theorem \cite{bib:FK}:

\begin{thm}
If $f$ is an automorphism of the variety $Z$, then the number of fixed
points of $f$ is given by the following formula:
\[
Fix(f, Z) = \sum^{2n}_{i = 0} (-1)^i \Trace f^{*}(H^{i}(Z)).
\]
\end{thm}

In the case $Z$ is defined over $\FF_p$ and $f = \Frob_p$, $Fix(f, Z)$
is simply the number of points of $Z$ over $\FF_p$.  We see that the
number of points of $Z$ over $\FF_p$ is related to the action of
$\Frob_p$ on the \'etale cohomology of $Z$.  
 
\begin{prop}
For $p$ congruent to 1 modulo 10, the semisimplification of the action of the Frobenius map $\Frob_p$ on
$H^2(\hat{X} \times_{\ZZ} \Fpbar, \QQ_l)$ is multiplication by $p$.
\end{prop}

\begin{proof}  We prove this statement in two steps.

\begin{step1}
The Frobenius map $\Frob_p$ acts on $H^2(\tilde{X}
\times_{\ZZ} \Fpbar, \QQ_l)$ by
multiplication by $p$.
\end{step1}

Here we recall an argument from \cite{bib:Lee}.  Consider the embedding $i: \PP^4 \times \PP^4 \rightarrow \PP^{24}$
that sends $(x,z)$ to $y$ where $y_{5i+j} = x_i z_j$.  Recall that
$\tilde{X}$ is a section of $\PP^4 \times \PP^4$ by five divisors of type
$(1,1)$; correspondingly $i(\tilde{X})$ is a section of $i(\PP^4 \times
\PP^4)$ by five hyperplanes; let $X^i$ denote
successive sections of $\PP^4 \times \PP^4$ by these hyperplanes.  Moreover, by Bertini's theorem these sections can be chosen such that $X^{i-1} -
X^i$ is smooth for all $i$.

By the Lefschetz theorem in \'etale cohomology, 
$$i^{*}: H^2(\PP^4 \times
\PP^4, \QQ_l) \rightarrow H^2(\tilde{X} \times_{\ZZ} \Fpbar, \QQ_l)$$ 
is an isomorphism that preserves the
Frobenius action.

Since all of $H^2(\PP^4 \times \PP^4, \QQ_l)$ can be
represented by divisors defined over $\FF_{31}$, the Frobenius map
acts by multiplication by $p$ on $H^2(\PP^4 \times \PP^4, \QQ_l)$.
Thus the Frobenius map acts likewise on $H^2(\tilde{X} \times_{\ZZ} \Fpbar,
\QQ_l)$.  Note that Step 1 is valid for any prime $p$, not just those
congruent to 1 modulo 10.

\begin{step2}  The semisimplification of the action of the Frobenius map $\Frob_{p}$ on
$H^2(\hat{X} \times_{\ZZ} \Fpbar, \QQ_l)$ is multiplication by $p$.
\end{step2}

The following argument is taken from \cite{bib:HV}.  Recall that $\pi: \hat{X} \rightarrow \tilde{X}$ is a blowup of
65 ordinary double points.  From the Leray spectral sequence for
$\pi$, we obtain an exact sequence

\begin{equation}
0 \longrightarrow H^2(\tilde{X} \times_{\ZZ} \Fpbar, \QQ_l) \longrightarrow H^2(\hat{X}
\times_{\ZZ} \Fpbar, \QQ_l) \longrightarrow
\oplus^{65}_{i = 1} H^2(Q_i, \QQ_l),
\end{equation}

\noindent where the $Q_i$ are the exceptional divisors.  For $p \equiv
1$ (mod 10), the
rulings on all the $Q_i$ are defined over $\FF_p$.  Hence the
Frobenius map acts by multiplication by $p$ on the $Q_i$.  From the
exact sequence, we see that the semisimplification of the Frobenius map acts by multiplication
by $p$ on $H^2(\hat{X} \times_{\ZZ} \Fpbar, \QQ_l)$.
\end{proof}

We will now concentrate our attention on the prime $p = 31$.  Let us collect the information we have so far about the cohomology of
$\hat{X} \times_{\ZZ} \overline{\FF}_{31}$:

\begin{enumerate}
\item $h^0 = h^6 = 1.$

\item $h^1 = h^5 = 0.$

\item The semisimplification of the  $\Frob_{31}$ action on $H^2$ is
  multiplication by 31.  By Poincar\'e duality, the semisimplification
  of the $\Frob_{31}$ action on $H^4$ is multiplication by ${31}^2$.

\item $2 h^2 - h^3 = 158.$

\item $ \# X(\FF_{31}) = 1 + 31 h^2 + {31}^2 h^2 + {31}^3 - \Trace
\Frob_{31}(H^3(\XT \times_{\ZZ} \overline{\FF}_{31}, \QQ_l))$.

\item For primes $p$ of good reduction, $| \Trace \Frob_{p}(H^3(\XT
  \times_{\ZZ} \Fpbar, \QQ_l)) | \leq h^3
{p}^{\frac{3}{2}} = (158 - 2 h^2) {p}^{\frac{3}{2}}$ by the Weil conjectures.
\end{enumerate}

The number of points in $\hat{X}$ is easily computed by the
following procedure:

\begin{enumerate}

\item For a given prime $p$, count the number of points in $X'$ using a
computer.

\item Add $p$ times the number of points in $E$, since each point in
$E$ is replaced by a copy of $\PP^1$ upon passage to $\tilde{X}$.

\item Add the number of points arising from the blowup of the nodes.

\end{enumerate}

The $\sigma$-nodes are
defined for all $\FF_p$, and the rulings over the exceptional divisors
exist over all $\FF_p$.  Hence each of these five nodes adds $p^2 + p$
points to the total.

The node $(1:1:1:1:1) \times (1:1:1:1:1)$ is defined over all $\FF_p$, but
the other $\tau$-nodes in its orbit are defined only if $p \equiv 1 (\Mod
5)$.  The rulings over the exceptional divisors exist over $\FF_p$
only if $\sqrt{5}$ is defined in $\FF_p$, i.e. if $p \equiv \pm 1
(\Mod 5)$.  If the rulings are defined over $\FF_p$, we add $p^2 + p$
points.  Otherwise we add only $p^2$ points.

The node $(2:-1:0:0:-1) \times (0:1:0:0:-1)$ is defined over all $\FF_p$ as are the other nodes in its $\sigma$-orbit.  However,
the other nodes in its $H_5$-orbit are defined over $\FF_p$
only if $\epsilon$ is defined.  If the nodes are defined,
then the rulings over the exceptional divisors are as well.  Hence we
add $p^2 + p$ times the number of these nodes.  A similar argument applies to $(2:-1:0:0:-1) \times (0:-2:1:2:2)$.

For $p = 31$, we obtain $\# X(\FF_{31}) = 110010$.  The integrality of $h^2$ and the inequalities imposed by the Weil conjectures force $h^2$ to be
equal to 81.  We then get $h^3 = 4$.  (This trick for computing $h^2$ is due to Werner and van
Geemen in \cite{bib:WvG}.)

For $p \not\equiv 1 (\Mod 10)$, we no longer know that the
semisimplification of $\Frob_p$ acts by
multiplication by $p$ on $H^2$.  However, we know that $\Frob_p$ acts on
$H^2(X)$ by multiplication by $p$.  We also know that
$\oplus^{65}_{i = 1} H^2(Q_i, \QQ_l)$ is spanned by algebraic cycles,
so the eigenvalues of $\Frob_p$ acting on this space are all $p$ times
roots of unity.  Using the exact sequence (29) again, the
eigenvalues of the semisimplification of $\Frob_p$ acting on
$H^2(\hat{X})$ are all $p$ times roots of unity.  By the Weil conjectures,
the trace of $\Frob_p$ is a rational integer.  Hence the trace of
$\Frob_p$ must be $p$ times an integer $h$.  

Suppose the eigenvalues of $\Frob_p$ acting on $H^2(\hat{X})$ are $p
\zeta_i$, with the $\zeta_i$ being roots of unity.  Choosing a basis
of $H^2(\hat{X})$ and a Poincar\'e dual basis of $H^4(\hat{X})$, the action of
$\Frob_p$ on $H^2(\hat{X})$ can be
represented as a matrix.  This matrix is similar to a matrix of upper
Jordan blocks having diagonal entries $p \zeta_i$.  By Poincar\'e
duality, the action of $\Frob_p$ acting on $H^4(\hat{X})$ will be $p^2$
times the contragredient of the action on $H^2(\hat{X})$.  Thus the matrix
of $\Frob_p$ acting on $H^4(\hat{X})$ will be similar to a matrix
of lower triangular blocks having diagonal entries $p^2 \overline{\zeta_i}$ with the
same multiplicities as in $H^2(\hat{X})$.  Hence the trace of $\Frob_p$ acting on $H^4$ is $hp^2$.

Furthermore, we have the additional piece of information that $h^2 =
81$, so we can use the Weil conjectures again.  For all primes $p \geq 29$ and congruent to $\pm 3 \: (\mod 10)$ for which we counted points, we found that $h = 33$.

\begin{conj}
For primes $p$ of good reduction, the trace of $\Frob_p$ acting on
$H^2(\hat{X})$ is $81p$ or $33p$, depending on whether or not $p$ is conrguent to 1 mod 10.
\end{conj}

There is a unique normalized cusp form $f$ of level 55 and weight 4, whose Fourier
coefficients $a_p$ can be found in William Stein's Modular Forms
Database \cite{bib:Stein}.  The first few coefficients of the
$q$-expansion of $f$ are as follows:

\begin{equation}
f = q + q^2 - 3q^3 - 7q^4 - 5q^5 - 3q^6 - 9q^7 - 15q^8 - 18q^9 - 5q^{10} + \dots
\end{equation}

\subsection{Proof of the main theorem}

We can finally prove the main theorem of this paper:

\begin{thm}
The third cohomology group $H^3$ of $\hat{X}$ is modular in
the following sense:  as a Galois representation, its
semisimplification is a direct sum of the 2-dimensional Galois representations $V$ and $H^1(E_2, \QQ_l)$ where $V$ and $W$ are associated to the modular forms $55k4A1$ and $550K1$ respectively.  Here $55k4A1$ has weight 4 and level 55, and $550K1$ has weight 2 and level 550.
\end{thm}

\begin{proof}  Denote by $S$ the preimage of $E_2$ in $\tilde{X}$; it is a ruled surface over $E_2$.  Denote by $\hat{S}$ the blowup of $S$ along the Heisenberg orbit of $(-2:1:0:0:1) \times (0:-2:1:-1:2)$.

Since we will only consider the coefficient group
$\QQ_l$, we can use the proper-smooth base change theorem (see for
example \cite{bib:Milne}) to pass from $\hat{X}$
to $\hat{X} \times_{\ZZ} \Fpbar$ (for $p \neq 2, 5, 11, l$) and to $\hat{X} \times_{\ZZ} \CC$.
In addition, we can pass from \'etale cohomology on $\hat{X} \times_{\ZZ} \CC$ to
analytic cohomology by the comparison theorem.

We present the proof in several steps.  Once again our arguments closely follow those in \cite{bib:Lee}.

\begin{step1}  $H^1(\hat{S} \times_{\ZZ} \Fpbar, \QQ_l)(-1) \cong
H^1(E_2 \times_{\ZZ} \Qpbar) \times_{\Qpbar} \Fpbar, \QQ_l)(-1)$.
\end{step1}

For now, we use the analytic topology.  Recall that the map $\pi: \hat{S}
\longrightarrow S$ is just the blowup at 25 points.  A standard Mayer-Vietoris 
argument (see for example \cite{bib:GH}, p. 473) shows that 
$H^1(\hat{S} \times_{\ZZ} \Fpbar, \QQ_l) \cong H^1(S \times_{\ZZ} \Fpbar, \QQ_l)$.

The map $\pi: S \longrightarrow E_2$ is the projectivization of the rank 2
bundle $\ker L \longrightarrow E_2$.  Using the analytic topology for
now, the Leray-Hirsch theorem tells us that
$H^{*}(S, \QQ_l)$ is a truncated polynomial ring over $H^{*}(E_2, \QQ_l)$ generated by
the single element $c_1(\ker L)$, which has dimension 2.  Hence
$\pi^{*}: H^1(E_2, \QQ_l) \longrightarrow H^1(S, \QQ_l)$ is an isomorphism.
This statement also holds in \'etale cohomology.

\begin{step2}  $H^1(\hat{S} \times_{\ZZ} \Fpbar, \QQ_l)(-1)$ is a subrepresentation
of $H^3(\hat{X} \times_{\ZZ} \Fpbar, \QQ_l)$.
\end{step2}

Note that $H^1(\hat{S} \times_{\ZZ} \Fpbar, \QQ_l)(-1)$ is isomorphic to
$H^3_{\hat{S}}(\hat{X} \times_{\ZZ} \Fpbar,
\QQ_l)$ (see \cite{bib:Milne}, p. 98).  So it is sufficient to show that inclusion induces an
injection 

\[
j^{*}: H^3_{\hat{S}}(\hat{X} \times_{\ZZ} \Fpbar, \QQ_l) \longrightarrow
H^3(\hat{X} \times_{\ZZ} \Fpbar, \QQ_l).
\]

By base change and the comparison theorem, it suffices to prove the
same statement in the complex analytic topology.

In the analytic topology, we have the Lefschetz duality diagram (see
\cite{bib:Munkres}, p. 429)

\begin{equation*}
\begin{CD}
H^3_{\hat{S}}(\hat{X} \times_{\ZZ} \CC, \QQ_l) @>{j^{*}}>> H^3(\hat{X} \times_{\ZZ} \CC, \QQ_l) \\
@VV{\cong}V                                @VV{\cong}V \\
H_3(\hat{S} \times_{\ZZ} \CC, \QQ_l)  @>{j_*}>> H_3(\hat{X} \times_{\ZZ} \CC, \QQ_l)
\end{CD}
\end{equation*}

\noindent where the vertical arrows are isomorphisms.  Hence it is sufficient to show that inclusion induces an injective map
$j_{*}: H_3(\hat{S} \times_{\ZZ} \CC, \QQ_l) \longrightarrow H_3(\hat{X} \times_{\ZZ}
\CC, \QQ_l)$.  We do this by computing
intersection classes of cycles in $H_3(\hat{S} \times_{\ZZ} \CC, \QQ_l)$.

The 3-cycles in $\hat{S}$ are of the form $\alpha \times \PP^1$ and
$\beta \times \PP^1$.  Note that the $\alpha$ and $\beta$ can be chosen to
miss the exceptional divisors.

We have the fundamental result

\[
j_*(\alpha) \cap j_*(\beta) = j_*( PD[\hat{S}]|_{\hat{S}} \cap \alpha
\cap \beta)
\]

\noindent for any homology cycles $\alpha$ and $\beta$.

Note that

\[
[\hat{S}] |_{\hat{S}} = (K_{\hat{X}} - K_{\hat{X}} + \hat{S}) | {\hat{S}}
= K_{\hat{S}} - K_{\hat{X}} | {\hat{S}}.
\]

Since $\tilde{X}$ has trivial canonical bundle.  Hence $K_{\hat{X}}$ is supported
on its exceptional fibers.  Therefore the restriction of $K_{\hat{X}}$ to
$\hat{S}$ is supported on the exceptional fibers of $\hat{S}$.  Let $\gamma$
and $\delta$ be 1-cycles generating the first homology of $E_2$, and
put $\alpha = \gamma \times C$ and $\beta= \delta \times C$,
where $C$ is a ruling of $S$.

Since the cycles $\alpha$ and $\beta$ can be chosen to miss the
exceptional fibers, we have

\[
PD[\hat{S}]|_{\hat{S}} \cap \alpha \cap \beta = PD
[K_{\hat{S}}] \cap \alpha \cap \beta.
\]

The canonical bundle of $\hat{S}$ is well-known; it is simply
$-2(D) + \Sigma E_i$, where $D$ is a horizontal section and the $E_i$
are the exceptional divisors.  We also have

\[
\gamma \cap \gamma = \delta \cap \delta = 0,
\]

\[
\gamma \cap \delta = C,
\]

\noindent where $C$ is a line belonging to the ruling of $S$. 

Therefore we have

\[
j_*(\alpha) \cap j_*(\alpha) = j_*(\beta) \cap j_*(\beta) = 0
\]

\[
j_*(\alpha) \cap j_*(\beta) = -2.
\]

Putting these results together, we see that the intersection matrix of
the 3-cycles is 

\[
\begin{pmatrix} 0 & -2 \\ 2 & 0 \\
\end{pmatrix},
\]

\noindent which is nonsingular.  Hence $H_3(\hat{S})$ injects into $H_3(\hat{X})$.

Upon passing to the semisimplification, we now see that as Galois representations,

\[
H^3(\hat{X} \times_{\ZZ} \Fpbar, \QQ_l) = V \oplus H^1(E_2 \times_{\ZZ}
\Fpbar, \QQ_l)(-1),
\]

\noindent with $V$ some undetermined 2-dimensional piece. 

\begin{step3}  Away from the primes of bad reduction, the
traces of $V$ coincide with the coefficients of the unique normalized modular cusp form $f$ of level 55 and weight 4; thus the semisimplifications of the associated Galois representations are isomorphic away from the primes of bad reduction.
\end{step3}

In the case where the traces of $V$ and $f$ are even, one can use the powerful method of Faltings-Serre-Livn\'e to find a finite set of primes $T$ such that equality of traces at the primes in $T$ guarantees the isomorphicity. \cite{bib:Livne}  However, in our example the traces are not all even, so we must use another method.

We can compute the traces of $V$ at the primes $29, 31, 37$.  By a method originally due to Serre in \cite{bib:Serre} and worked out in detail by Schu\"tt \cite{bib:Schutt}, if one knows the set $S$ of bad primes and the set of Galois number field extensions of small degree unramified outside $S$, then one can often find a finite set $T$ of primes such that the agreement of traces at the primes of $T$ guarantees the isomorphicity of the Galois representations.  The fact that the traces of $V$ and $f$ are odd at $29, 31$ and $37$ imply that we need only check the equality of traces at $43, 47, 59, 83$ to prove that the two representations are isomorphic.  We are able to calculate the traces of $V$ at the set of primes $\{29, 31, 37, 43, 47, 59, 83 \}$ as shown in the table below and check that they agree with the coefficients of $f$.

\begin{Table}[h]
\begin{center}
\begin{tabular}{c|c|c}
$p$ & $\# \hat{X}(\FF_p)$ & $\trace V$ \\

\hline 

29 & 53120 & -165 \\

31 & 110010 & -83 \\

37 & 97310 & 1 \\

43 & 142210 & -8 \\

47 & 177770 & 126 \\

59 & 322490 & -290 \\

83 & 800690 & 842 \\
\end{tabular}
\end{center}
\end{Table}

We outline the argument of Serre-Schu\"tt below.  Start with two $l$-adic Galois representations

\[
\rho_i: \Gal(\QQ) \rightarrow GL_n(\QQ_l)
\]

both unramified outside a finite set of primes $S$; after specifying a stable lattice we may assume that they take values in $GL_n(\ZZ_l)$ (cf. \cite{bib:Serre2}).  We now impose the following requirements:

\begin{enumerate}
\item The $\rho_i$ have the same determinant.

\item The mod $l$ reductions $\overline{\rho}_i$ are absolutely irreducible and isomorphic.
\end{enumerate}

Obviously the traces of the $\rho_i$ are the same mod $l$.  Assume that there is some prime $p \notin S$ such that $\trace \rho_1(\Frob_p) \neq \trace \rho_2(\Frob_p)$, and choose the maximal $\alpha \in \ZZ^{+}$ such that the $\rho_i$ are isomorphic modulo $l^\alpha$.  We construct a map $\tau$ which measures the non-isomorphicity of the $\rho_i$:

\begin{equation*}
\begin{aligned}
\tau: \Gal(\QQ) &\rightarrow \FF_l \\ 
\sigma &\rightarrow \frac{\trace \rho_1(\sigma) - \trace \rho_2(\sigma)}{l^\alpha} \mod l \\
\end{aligned}
\end{equation*}

The map $\tau$ maps the inertia groups $I_p$ to 0 for $p \notin S$.  We now construct a factorization $\tau = \tilde{\tau} \circ \tilde{\rho}$ and investigate the map $\tau$.

After replacing $\rho_1$ with a conjugate if necessary, we may assume that $\rho_1 \cong \rho_2 \mod l^\alpha$.  Hence for every $\sigma \in \Gal(\QQ)$ there is a matrix $\mu(\sigma) \in M_n(\ZZ_l)$ such that

\[
\rho_1(\sigma) = (1 + l^\alpha \mu(\sigma)) \rho_2(\sigma).
\]

Since this relation describes $\tau(\sigma)$ as $\trace \mu(\sigma) \rho_2(\sigma)$, we factor $\tau$ through the product $M_n(\ZZ_l) \times GL_n(\ZZ_l)$.  Due to the definition of $\tau$ mod $l$, we can restrict ourselves to the product of the $\mod l$ reductions $\overline{\mu}$ and $\overline{\rho}_2$.  We can then turn the map $\tilde{\rho} = \overline{\mu} \times \overline{\rho}$ into a group homomorphism by giving the target set $M_n(\FF_l) \times GL_n(\FF_l)$ the structure of a semidirect product product with operation

\[
(A,C) \cdot (B,D) = (A + CBC^{-1}, BD).
\]

By construction, $\tilde{\rho}$ is unramified outside $S$.  Furthermore, condition (1) implies that $\det(1 + l^\alpha \mu) = 1$.  Expanding in powers of $l$, we have $1 = 1 + l^\alpha \trace \mu +\l^{2\alpha} (\dots)$.  Hence $\trace \mu \cong 0$ mod $l$ and $\tilde{\rho}$ maps  $\Gal(\QQ)$ into $\tilde{G}$, where $\tilde{G}$ is the subgroup of $M_n(\FF_l) \times GL_n(\FF_l)$ consisting of elements where the first matrix has trace zero $\mod l$.

In our specific situation, set $n = 2$ and $l = 2$.  We construct the map $\tilde{\rho}$ in the case where the $\rho_i$ do not have even trace but have the same determinant $\chi_2^3$.  To check that condition (2) holds, we need to compute the Galois extensions $K_i/\QQ$ cut out by the kernels of the mod 2 reductions $\overline{\rho}_i$.  The absolute irreducibility of $\overline{\rho}_i$ is equivalent to $K_i/\QQ$ having Galois group $S_3$ in our situation, since the traces not being even implies that the $K_i/\QQ$ have Galois group $S_3$ or $C_3$, the image of $\overline{\rho}_i$ in $GL_n(\FF_2)$.  As the $K_i$ are unramified outside $S$, there are finitely many possible isomorphism classes of them, all of which can be computed by class field theory; we find them in the tables of J. Jones \cite{bib:Jones}.  The number fields $K_i$ cut out by $\overline{\rho}_i$ can then be determined by the fact that $\Frob_p$ has order 3 in $\Gal(K_i/\QQ)$ if and only its trace is odd.  Hence the simultaneous oddness of the traces of $\rho_i(\Frob_p)$ for some suitable set of primes $T$ guarantees the isomorphicity of the mod $2$ reductions $\overline{\rho}_i$.

With conditions (1) and (2) satisfied, we now construct the map $\tilde{\rho}$ sketched above under the assumption (to be contradicted) that the $\rho_i$ are not isomorphic.  With this assumption, the map $\tau$ must be nonconstant.  Now consider the map

\begin{equation*}
\begin{aligned}
\tilde{\tau} : \tilde{G} &\rightarrow \FF_2 \\
(A,C) &\rightarrow \trace (AC) \mod 2 \\
\end{aligned}
\end{equation*}

which gives $\tau = \tilde{\tau} \circ \tilde{\rho}$.

One can show that the target group $\tilde{G}$ is isomorphic to $S_4 \times C_2$.  By direct computation, $\tilde{\tau}$ is nonzero exactly at the elements of $\tilde{G}$ of order greater than 3; hence $\image \tilde{\rho}$ must contain an element of order 4 or 6.  Since $\image \overline{\rho} \cong S_3$ is contained in $\image \tilde{\rho}$, the image of $\tilde{\rho}$ must contain $S_3 \times C_2$ or $S_4$.  By Galois theory, $\image \tilde{\rho}$ corresponds to a number field $L$ unramified outside $S$ with intermediate field $K = K_1 = K_2$, the extension cut out by $\ker \overline{\rho}$.  Again, Galois extensions with sufficiently small Galois group and ramification locus have been classified, and we will rule out each individual possibility for $L$ as follows:  Find $\Frob_p$ with maximal order in $\Gal(L/\QQ)$; this implies $\tilde{\tau}(\Frob_p) = 1$.  Then check the explicit equality$\trace \rho_1(\Frob_p) = \trace \rho_2(\Frob_p)$ at $p$, thus ruling out $L$.

In our situation, the 2-adic Galois representations $\rho_1$ and $\rho_2$ associated to the motive $V$ and the modular form $f$ are unramified outside $\{2,5,11\}$ and have determinant $\chi_2^3$.  By the tables of Jones, the oddness of the traces at 29, 31 and 37 forces the intermediate field $K$ cut out by the kernels of the mod 2 reductions to be $\QQ(x^3 + 2x - 8)/\QQ$ with Galois group $S_3$ and intermediate field $\QQ(x^2 + 110)$, establishing condition (2) above. 

If $\image \tilde{\rho}$ were to contain an element of order 6, we would have a Galois extension $L/\QQ$ with intermediate field $K$ and Galois group $S_3 \times C_2$.  $S_3 \times C_2$ evidently has quotient groups isomorphic to $C_2$, and all possible quadratic field extensions $J/\QQ$ unramified outside $S$ are known.  In all cases the generating polynomials of the quadratic extensions are irreducible modulo at least one of 29, 31 or 37.  Hence the corresponding Frobenius element of at least one of the primes 29, 31 or 37 must have order 6 in $\tilde{\rho}$.  This is impossible because the traces of $\rho_1$ and $\rho_2$ agree at all three primes.

In order to rule out elements of order 4 in $\image \tilde{\rho}$, we need to know all Galois extensions $L/\QQ$ with Galois group $S_4$ and intermediate field $\QQ(x^3 + 2x - 8)$; by calculations of Jones, these fields have generating polynomials

\begin{center}
\begin{tabular}{lll}
$x^4 - 220x + 165$ & $x^4 - 2x^3 - 26x^2 - 28x - 24$ & $x^4 - 10x^2 - 40x + 10$ \\
$x^4 - 2x^3 - x^2 - 8x - 4$ & $x^4 - 2x^3 - 15x^2 - 28x - 24$ & $x^4 - 22x^2 - 176x - 110$ \\
$x^4 - 440x - 3410$ & $x^4 - 20x^2 - 40x - 10$ & $x^4 + 20x^2 - 80x + 60$ \\
$x^4 + 44x^2 - 352x + 132$ & $x^4 - 12x^2 - 16x - 20$ & $x^4 - 8x - 2$ \\
$x^4 - 4x^2 - 4x + 9$ & $x^4 - 176x + 418$ & $x^4 - 1760x - 440$ \\
\end{tabular}
\end{center}

Each of these polynomials is irreducible modulo at least one of the primes 43, 47, 59, 83.  Hence for each Galois extension the Frobenius element at one of these primes has order 4, which is impossible because the traces of $\rho_1$ and $\rho_2$ agree at all four primes.  Hence the equality of traces at the set of primes $\{29, 31, 37, 43, 47, 59, 83 \}$ guarantees the isomorphicity of $\rho_1$ and $\rho_2$.  \end{proof}

\begin{cor}
Up to Euler factors at the primes of bad reduction, the $L$-function $L(\hat{X},s)$ of the middle cohomology of $\hat{X}$ is equal to $L(f,s) L(g,s-1)$ where $f$ is the unique normalized cusp form of weight 4 and level 55, $g$ is the normalized cusp form $550K1$ of weight 2 and level 550 from Stein's database, and $L(f,s)$ and $L(g,s)$ are the Mellin transforms of $f$ and $g$.
\end{cor}

We end with the following question:

\begin{question}  Is our threefold $\hat{X}$ birational to either of the threefolds in \cite{bib:Schutt} $\tilde{W}_1$ of level 55 or $\tilde{W}_2$, also conjectured to be modular of level 55?  If not, is there a correspondence between $\hat{X}$ and either threefold?
\end{question}

\vfill
\pagebreak


\end{document}